\documentclass{amsart}
\usepackage{amsmath}
\usepackage{amsfonts}
\usepackage{amssymb}
\usepackage{graphicx}%

\newtheorem{myrem}{Remark}[section]
\newtheorem{mytheo}{Theorem}[section]
\newtheorem{mylem}{Lemma}[section]

\usepackage{color}

\theoremstyle{remark}

\numberwithin{equation}{section}

 \allowdisplaybreaks[4]
\begin{document}

\title{Tail Bounds on the Spectral Norm of Sub-Exponential Random Matrices}


\author{Guozheng Dai}
\address{School of Mathematical Sciences,  Zhejiang University, Hangzhou, 310027,  China.}
\email{11935022@zju.edu.cn}

\author{Zhonggen Su}
\address{School of Mathematical Sciences, Zhejiang University, Hangzhou, 310027,  China.}
\email{suzhonggen@zju.edu.cn}
\thanks{Zhonggen Su was supported  by  the National Natural Science Foundation of China (No.12271475 and No. 11871425) and fundamental research funds for central universities grants.}

\author{Hanchao Wang }
\address{ Institute for Financial Studies, Shandong University,  Jinan,  250100, China.}
\email{wanghanchao@sdu.edu.cn}
\thanks{Hanchao Wang  was supported by  the National Natural Science Foundation of China (No. 12071257 and No. 11971267 );   National Key R$\&$D Program of China (No. 2018YFA0703900 and No. 2022YFA1006104); Shandong Provincial Natural Science Foundation (No. ZR2019ZD41).
}

\date{}


\begin{abstract}
Let $X$ be an $n\times n$ symmetric random matrix with independent but non-identically distributed entries. The deviation inequalities of the spectral norm of $X$ with Gaussian entries have been obtained by using the standard concentration of Gaussian measure results.  This paper establishes an upper tail bound of the spectral norm of $X$ with sub-Exponential entries.  Our   method relies upon a crucial ingredient of a novel chaining argument that  essentially involves both the particular structure of the sets used for the chaining and  the distribution of coordinates of a point on the unit sphere.
\end{abstract}

\maketitle

\section{Introduction and Main Result}
Random matrix theory has been  a rapidly developing area of probability theory in the past decades. Much  work has been done  about matrices with  exact or approximate symmetries, such as matrices with i.i.d. entries, for which precise analytic results and limit theorems are available. We refer the readers to \cite{introduction,2015Random,Tao.matrices} and the references therein for many fundamental  probability limit theorems with respect to eigenvalue statistics.

More recently, there has also been considerable interest in structured random matrices where the entries are no longer identically distributed.  Let   $X$ be the $n\times n$ symmetric random matrix with entries $X_{ij}=b_{ij}g_{ij}$, where $\{g_{ij}: i\ge j\}$ are independent standard Gaussian random variables  and $\{b_{ij}: i\ge j\}$ are deterministic nonnegative scalars. The structure of $X$ is controlled by the given variance pattern $\{b_{ij}, i\ge j\}$.  Note that it does not make much sense to investigate their asymptotic properties since the structure is  only defined for the finite  matrices and so there is no natural way to take the matrix size to infinity.

Due to this observation, the study of structured random matrices has a significantly different flavor than most classical random matrix theories. Instead,  the primary interest in this area is to obtain nonasymptotic  probability  inequalities that identify what structural parameters control the macroscopic properties of the underlying random matrix. In this paper we are interested in the location of the edge of the spectrum, that is, the matrix's spectral norm $\Vert X\Vert$,  and particularly focus on probabilistic inequalities (upper tail bounds) on $\Vert X\Vert$.

As  known to us, the mathematical expectation of a random variable can be expressed in terms of  its tail probability by the integration formula by parts. In particular, we have
\begin{align}\label{1}
	\mathbb{E}\Vert X\Vert=\int_{0}^{\infty} \mathbb{P}\{\Vert X\Vert>t\}\,dt.
\end{align}
 Due to this  identity, many authors study the bounds of $\mathbb{E}\Vert X\Vert$ and $\mathbb{P}\{\Vert X\Vert \ge t\}$  simultaneously. On the other hand,  it follows from Gaussian concentration inequality (see Theorem 5.6 in \cite{concentration}) that $\mathbb{P}\{\Vert X\Vert-\mathbb{E}\Vert X\Vert\ge t\}\le e^{-t^{2}/2}$. So one can easily derive  a tail probability bound from the expectation $\mathbb{E}\Vert X\Vert$. For the sake of reading, we give a very simple review of the relevant results in the case of sub-Gausssian entries.

Let $X=(b_{ij}g_{ij})_{n\times n}$ as above. Vershynin \cite{highdimension} proved that the spectral norm of $X$ is bounded by   $\sqrt{n}$ up to a factor with high probability, that is,
\begin{align}
	\mathbb{P}\{ \Vert X\Vert\ge C\max_{ij}b_{ij}\sqrt{n}+t\}\le \exp\Big(-\frac{t^{2}}{\max_{i,j}b_{ij}}\Big).\nonumber
\end{align}  It follows from (\ref{1})
\begin{align}\label{2}
	\mathbb{E}\Vert X\Vert\lesssim \max_{ij}b_{ij}\sqrt{n},
\end{align}
where and in the sequel we write $a\lesssim b$ if $a\le Cb$ for a universal constant $C$. In fact, Vershynin proved these results when the entries are independent sub-Gaussian variables. Although this bound capture the correct $\sqrt{n}$ rate for Wigner matrices with i.i.d. entries, it fails to be sharp in other  cases. For example, consider the diagonal matrix example with entries are independent standard Gaussian variables, then $\mathbb{E}\Vert X\Vert\sim\sqrt{\log n}$.

Bandeira and van Handel \cite{aop} obtained a better bound of $\mathbb{E}\Vert X\Vert$ than (\ref{2}), namely
\begin{align}\label{3}
	\mathbb{E}\Vert X\Vert\lesssim \max_{i}\sqrt{\sum\limits_{j}b_{ij}^{2}}+\sqrt{\log n}\max_{ij}b_{ij}.
\end{align}
This inequality is optimal in a surprisingly general setting, including Wigner matrices and diagonal matrices.  Banderia and van Handel \cite{aop} also studied (\ref{3}) for heavy-tailed entries and even bounded entries. Furthermore, they   considered the tail bounds for $\Vert X\Vert$ and obtained for the Gaussian  entries
\begin{align}
	\mathbb{P}\Big\{\Vert X\Vert\ge C(\max_{i}\sqrt{\sum_{j}b_{ij}^{2}}+\max_{i,j}b_{ij}\sqrt{\log n})+t\Big\}\le e^{-t^{2}/4\max_{i,j}b_{ij}^{2}}\nonumber.
\end{align}
And they   developed a handy tail bound of $\Vert X\Vert$ when the entries are bounded.

Only recently, did Latala, van Handel and Youssef \cite{inventions} obtain a sharp bound for $\Vert X\Vert$
\begin{align}\label{4}
	\mathbb{E}\Vert X\Vert\asymp \max_{i}\sqrt{\sum\limits_{j}b_{ij}^{2}}+\max_{ij}b_{ij}^{*}\sqrt{\log i},
\end{align}
where and in the sequel we write $a\asymp b$ if $a\lesssim b$ and $b\lesssim a$. Here the matrix $\{b_{ij}^{*}\}$ is obtained by permuting the rows and columns of the matrix $\{b_{ij}\}$ such that
\begin{align}
	\max_{j}b_{1j}^{*}\ge \max_{j}b_{2j}^{*}\ge \cdots \ge \max_{j}b_{nj}^{*}.\nonumber
\end{align}
They in fact gave sharp bounds for Schatten $p$-norm of $X$ for any $2\le p\le \infty$, and (\ref{4}) is the special case of $p=\infty$. Similar to \cite{aop}, Latala, van Handel and Youssef \cite{inventions} also extended their results to the case where $X_{ij}$ are non-gaussian, and contained many other exciting things which we do not cover here.

In addition to the works mentioned above,  we   refer interested readers to \cite{free,tams,structure,ptrfvershynin, ejpcai,ejp,Jmlr,Talagrand,bernoullivershynin} for more information about the expectations and   tail bounds of the spectral norm of random matrices.

Throughout the remainder of this paper, $X$ will denote an $n\times n$ symmetric random matrix with independent but non-identically distributed centered sub-Exponential random entries in the absence of additional instructions. For a sub-Exponential random variable $\xi$, denote by $\Vert \xi \Vert_{\psi_{1}}$ the sub-Exponential norm, that is,
$$\Vert \xi\Vert_{\psi_{1}}=\inf \Big\{K>0; \mathbb{E}\exp(\frac{\vert \xi\vert}{K})\le 2\Big\}.$$
In such a  setting, we would prefer use the sub-Exponential norms  than   variances.  Note that the $\{\Vert X_{ij}\Vert_{\psi_{1}}\}$ actually controls the structure of $X$, just like $\{b_{ij}\}$ in the case of $X_{ij}=b_{ij}g_{ij}$.One may better understand this through Lemma 2.2 below.

Now we are ready to state our main result. For convenience, set
$$
 b =\max_{1\le i\le n}\Vert X_{ii}\Vert_{\psi_{1}}\log n, \quad \sigma_{1}=\max_{i, j\le n}\Vert X_{ij}\Vert_{\psi_{1}}, \quad \sigma_{2}=\max\limits_{i\neq j}\Vert X_{ij}\Vert_{\psi_{1}}.
$$

\begin{mytheo}
	The spectral norm of $X$ satisfies
	\begin{align}\label{5}
			\mathbb{P}\Big\{\Vert X\Vert\ge
		C(b+\sigma_{2}\sqrt{n})+t  \Big\}\le C_{0}\exp \Big(-C_{1}\min (\frac{t^{2}}{\sigma_{1}^{2}}, \frac{t}{\sigma_{1}})\Big),
	\end{align}
where $C, C_{0}, C_{1}$ are universal constants and $t\ge 0$.
\end{mytheo}

The proof of Theorem \ref{5} will be postponed to Section 4. To conclude the Introduction, we would like to briefly review the relevant literature and our basic ideas.

As mentioned above,  Banderia and van Handel \cite{aop} obtained  in the   Gaussian case the the expectation and   tail bounds  for  $\|X\|$ by applying standard Gaussian concentration techniques (first estaimate the bound for expectation and then control tail probability).   However, in our setting, i.e., sub-Exponential entries, due to the lack of the corresponding concentration inequalities, such an argument is not applicable even though  explicit bounds of $\mathbb{E}\Vert X\Vert$ is known.

Vershynin  \cite{introduction} (see Theorem 4.4.5) studied the tail bounds for $\|X\|$ by an $\varepsilon$-net argument for sub-Gaussian matrix $X$.  Specifically, he treated the spectral norm of $X$ as the
supremum of a stochastic process indexed by the Euclidean unit sphere $S^{n-1}$, that is
$$ \Vert X\Vert=\sup_{x,y \in S^{n-1}} \langle Xx, y\rangle,$$
and then approximated it by the maximum of $\langle Xx, y\rangle$ through a finite $\varepsilon$-net. However, if we simply follow his argument then we could only    obtain the following result for sub-Exponential matrix
\begin{align}\label{6}
	\mathbb{P}\{ \Vert X\Vert\ge C_{2}\sigma_{2}n\}\le e^{-n},
\end{align}
where $C_{2}$ is a universal constant.

On the other hand, Tao and Vu \cite{cmp} proved  that the largest eigenvalue is approximately the same order as $\sqrt{n}$ for independent but not identically distributed  entries with uniformly exponential decay when $n$ is large enough. Motivated by this work, we expect a tail bound with $\sigma_2\sqrt{n}$ for $\|X\|$ in (\ref{6}).

To prove Theorem \ref{5}, we shall adapt the stochastic process approach. Instead of using  directly an $\varepsilon$-net argument to an unit vector, we slice a point on the unit sphere into many, say $l(n)$, small parts, and then find a sufficiently good net to approximate every part. This idea is inspired by Adamczak, Litvak,  Pajor and Tomczak-Jaegermann \cite{jams}.

The rest of this article is organized as follows. In Section 2, we give some notations and lemmas used throughout the proofs. In Section 3, we give a concentration inequality for such a particular case like diagonal matrix as a warm-up. In Section 4, we complete the proof of Theorem \ref{5}. Section 5 partly extends our results to more general cases.

\section{Basic Lemmas}
In this section, we shall introduce some important lemmas that will be used to prove our main result. We first give some notations. We equip $\mathbb{R}^{n}$ with natural scalar product $\left\langle\cdot, \cdot \right\rangle$ and the natural Euclidean norm $\vert \cdot\vert$. We also denote by the same notation  $\vert \cdot\vert$ the cardinality of a set. By $\Vert M\Vert$ we shall denote the spectral norm of an $n\times n$ matrix $M$ as we mentioned before, that is, $\Vert M\Vert=\sup_{y\in S^{n-1}}\vert My\vert$, where $S^{n-1}=\{y\in \mathbb{R}^{n}: \vert y\vert=1\}$.
We will also use the following notation $\text{supp}\, y=\{i: y=(y(1),\cdots, y(n))'\in \mathbb{R}^{n}, y(i)\neq 0\}$.  Given a set $E\subset\{1,\cdots ,n\}$, by $P_{E}$ we denote the orthogonal projection from $\mathbb{R}^{n}$ onto the coordinate subspace of vectors whose support is in $E$. Such a subspace is denoted by $\mathbb{R}^{E}$.

 Let $B_{2}^{n}=\{x\in \mathbb{R}^{n}: \vert x\vert\le 1\}$ and $B_{\infty}^{n}=\{x=(x(1),\cdots,x(n))\in \mathbb{R}^{n}: \max\limits_{i} \vert x(i)\vert\le 1\}$. Given an $E\subset\{1,\cdots, n\}$ and $\varepsilon, \alpha \in (0,1]$, by $\mathcal{N}(E, \varepsilon, \alpha)$ we denote an $\varepsilon$-net of  $B_{2}^{n}\cap\alpha B_{\infty}^{n}\cap \mathbb{R}^{E}$ in the Euclidean metric. Note that $\vert \mathcal{N}(E, \varepsilon, \alpha )\vert\le (3/\varepsilon)^{\vert E\vert}$, which can be obtained from Corollary 4.2.13 in \cite{introduction}.

Throughout the paper, we will use $C, C_{0}, C_{1},\cdots$ to denote some universal positive constants (independent of $n$ and the random variables), which may differ from section to section. Moreover, we will use $C(\alpha)$ to denote some positive constants only depending on the parameter $\alpha$.

Next, we shall give the following lemma about sub-Exponential properties. We omit the proof for simplicity and refer interested readers to Proposition 2.7.1 of \cite{highdimension} for detailed proof.

\begin{mylem}[Sub-Exponential properties]
	Let $\eta$ be a sub-Exponential random variable. Then the following properties are equivalent.
	\begin{enumerate}
		\item $\mathbb{P}\{\vert \eta\vert\ge t \}\le 2\exp(-t/K_{1})$ for all $t\ge 0$;
		\item $(\mathbb{E}\vert \eta\vert^{p})^{1/p}\le K_{2}p$ for all $p\ge 1$;
		\item $\mathbb{E}\exp(\lambda\vert\eta\vert)\le \exp(K_{3}\lambda)$ for all $\lambda$ such that $0\le \lambda\le \frac{1}{K_{3}}$;
		\item $\mathbb{E}\exp(\vert\eta\vert/K_{4})\le2$.
		
		If $\mathbb{E}\eta=0$, we also have the following equivalent property:
		\item $\mathbb{E}\exp(\lambda X)\le \exp(K_{5}^{2}\lambda^{2})$  for all $\lambda$ such that $\vert \lambda\vert\le \frac{1}{K_{5}}$.
	\end{enumerate}
\end{mylem}
	We remark that the parameters $K_{1},\cdots,K_{5}$ that appeared in Lemma 2.1 are not universal constants. In fact, $K_{i}=C_{i}\Vert \eta\Vert_{\psi_{1}}$ for $i=1,\cdots, 5$, where $C_{1},\cdots ,C_{5}$ are universal constants. Moreover, $K_{i}\lesssim K_{j}$ for any $i, j\in \{1, 2, \cdots, 5\}$. One can refer to \cite{highdimension} for more information.

Variance and sub-Exponential norm are two essential parameters of sub-Exponential random variables. The following lemma illustrates the connection between them.

\begin{mylem}
	Let $\eta$ be a centered sub-Exponential random variable. Then we have $\mathbb{E} \eta^{2}\lesssim\Vert \eta\Vert_{\psi_{1}}^{2}$.
\end{mylem}
\begin{proof}
	It follows from the sub-Exponential properties 2 when $p=2$,
	$
		\mathbb{E}\eta^{2}\lesssim K_{2}^{2}
	$ ($K_{2}$ is the parameter that appeared in Lemma 2.1).
	Then we get the desired result directly.
\end{proof}

The following lemma, Bernstein's inequality, is a concentration inequality for sums of independent sub-Exponential random variables. One can refer to \cite{highdimension} for  proof.

\begin{mylem}
	Let $\eta_{1},\cdots,\eta_{n}$ be independent mean zero sub-Exponential random variables, and let $a=(a_{1},\cdots, a_{n})^{'}\in \mathbb{R}^{n}$. Then, for every $t\ge 0$, we have
	\begin{align}
		\mathbb{P}\{ \sum_{i=1}^{n}a_{i}\eta_{i}\ge t\}\le \exp\bigg(-C\min(\frac{t^{2}}{K^{2}\Vert a\Vert_{2}^{2}}, \frac{t}{K\Vert a\Vert_{\infty}})\bigg),\nonumber
	\end{align}
where $K=\max_{i}\Vert \eta_{i}\Vert_{\psi_{1}}$ and $C$ is a universal constant.
\end{mylem}

Then, we give the following lemma on the maximum of $n$ independent exponential distributed random variables.
\begin{mylem}
	Let $\eta_{1},\cdots, \eta_{n}$ be a sequence of independent exponential distributed random variables with parameter $\lambda =1$. Denote their order statistics by $\eta_{(1)}\le \cdots \le \eta_{(n)}$. Consider the following linear changes,
	\begin{align}
		T_{1}=2n\eta_{(1)}, T_{2}=2(n-1)(\eta_{(2)}-\eta_{(1)}),\cdots, T_{n}=2(\eta_{(n)}-\eta_{(n-1)}).\nonumber
	\end{align}
Then, we have $\{T_{i}\}$ are identically independent chi-square distributed random variables whose degrees of freedom are 2.\end{mylem}

\begin{myrem}
	(i) It follows from Lemma 2.4
	\begin{align}
	\eta_{(n)}=\sum_{i=1}^{n}\frac{T_{i}}{2(n-i+1)}.\nonumber
	\end{align}

Hence, we have
\begin{align}
	\mathbb{E}\eta_{(n)}=\mathbb{E}\sum_{i=1}^{n}\frac{T_{i}}{2(n-i+1)}=\sum_{i=1}^{n}\frac{1}{n-i+1}=\sum_{i=1}^{n}\frac{1}{i},\nonumber
\end{align}
which implies the expectation of $\eta_{(n)}$ is of order $\log n$.

Besides, it is readily seen by Lemma 2.3
\begin{align}
	\mathbb{P}\{\eta_{(n)}-\mathbb{E}\eta_{(n)}\ge t\}\le \exp(-C_{1}\min (t^{2}, t)),\nonumber
\end{align}
where $C_{1}$ is a universal constant and $t>0$.

At last, we have
\begin{align}
	\mathbb{E}e^{\lambda (\eta_{(n)}-\mathbb{E}\eta_{(n)})}=\prod_{i=1}^{n}\mathbb{E}\exp\left(\lambda \frac{T_{i}-2}{2(n-i+1)}\right)\le \exp\left(\lambda^{2}C^{2}\sum_{i=1}^{n}\frac{1}{i^{2}}\right),\nonumber
\end{align}
which means that $\Vert \eta_{(n)}-\mathbb{E}\eta_{(n)}\Vert_{\psi_{1}}\lesssim \sum_{i\le n}1/i^{2}$.

(ii)As mentioned above, one can represent $\eta_{n}$ as a sum of a sequence of independent variables. Note that
\begin{align}
	Var\Big(  \sum_{i=1}^{n}\frac{T_{i}}{2(n-i+1)}\Big)=\sum_{i=1}^{n}\frac{1}{i^{2}}\asymp \frac{\pi^{2}}{6}<\infty,\nonumber
\end{align}
not satisfying the condition of central limit theorem. Hence, the tail decay of $\eta_{n}-\mathbb{E}\eta_{n}$ is not $e^{-t^{2}}$.
\end{myrem}

Next, we shall give the following comparison theorem (Lemma 4.7 in \cite{inventions}).

\begin{mylem}
	Let $h_{i}$ and $h_{i}^{'}$, $i=1, \cdots, n$ be independent centered random variables such that
	\begin{align}
		(\mathbb{E}\vert h_{i}\vert^{p})^{1/p}\le c_{0}p^{\beta}, \qquad  c_{1}p^{\beta}	\le(\mathbb{E}\vert h_{i}^{'}\vert^{p})^{1/p}\le c_{2}p^{\beta}\nonumber
	\end{align}
	for all $p\ge 2$ and $i=1,\cdots, n$, where $c_{0}, c_{1}$ and $c_{2}$ are constants depending on random variables. Then there exists a constant $c$ depending only on $c_{0}, c_{1}, c_{2}$, and $ \beta$ such that
	\begin{align}
		\mathbb{E}f(h_{1},\cdots,h_{n})\le \mathbb{E}f(ch_{1}^{'},\cdots,ch_{n}^{'})\nonumber
	\end{align}
	for every symmetric convex function $f: \mathbb{R}^{n}\to \mathbb{R}$.
\end{mylem}

\begin{mylem}[Sub-Exponential maxima]
	Suppose that $\xi_{1}, \cdots, \xi_{n}$ are independent Exponential variables with parameter $1$. Let $\alpha_{1},  \cdots,  \alpha_{n}>0$ be a sequence of constants. Then we have
	\begin{align}
		\mathbb{E} \max_{i\le n}\vert \alpha_{i}\xi_{i}\vert\asymp \max_{i\le n}\alpha_{i}^{*}\log(i+1),\nonumber
	\end{align}
where $\alpha_{1}^{*}\ge \cdots\ge  \alpha_{n}^{*}$ is the decreasing rearrangement of $\alpha_{1},  \cdots,  \alpha_{n}$.
\end{mylem}

\begin{proof}
	Let $g_{1}, \cdots, g_{n}$ be a sequence of independent standard Gaussian variables. It follows from Lemma 4.5 of \cite{structure}
\begin{align}
	\mathbb{E} \max_{i\le n}\vert \sqrt{\alpha_{i}}g_{i}\vert\lesssim \max_{i\le n}\sqrt{\alpha_{i}\log(i+1)}.\nonumber
\end{align}

Let $g_{1}^{'}, \cdots, g_{n}^{'}$ be another sequence of independent standard Gaussian variables. Then,
\begin{align}
	\mathbb{E}\max_{i\le n}\vert \alpha_{i}g_{i}g_{i}^{'}\vert&=\mathbb{E}\bigg(\mathbb{E}\Big(\max_{i\le n}\vert \alpha_{i}g_{i}g_{i}^{'}\vert\big| (g_{1},\cdots ,g_{n})\Big)\bigg)\nonumber\\
	&\lesssim \mathbb{E}\max_{i\le n}\vert \alpha_{i}g_{i}\vert\sqrt{\log (i+1)}\nonumber\\
	&\lesssim \max_{i\le n}\alpha_{i}\log(i+1)\nonumber.
\end{align}

 Hence, we have by permutation invariance
\begin{align}
	\mathbb{E}\max_{i\le n}\vert \alpha_{i}g_{i}g_{i}^{'}\vert\lesssim \max_{i\le n}\alpha_{i}^{*}\log(i+1)\nonumber.
\end{align}

Note that for $p\ge 1$,
\begin{align}
	\big(\mathbb{E}\vert g_{i}g_{i}^{'}\vert^{p}\big)^{1/p}\asymp\big(\mathbb{E}\xi_{i}^{p}\big)^{1/p}.
\end{align}
It follows from Lemma 2.5
\begin{align}
		\mathbb{E} \max_{i\le n} \alpha_{i}\xi_{i}\lesssim\mathbb{E}\max_{i\le n}\vert \alpha_{i}g_{i}g_{i}^{'}\vert\lesssim \max_{i\le n}\alpha_{i}^{*}\log(i+1)\nonumber.
\end{align}

Next, we shall prove the inequality in the opposite direction. For a fixed $j\le n$, we have
\begin{align}
	\mathbb{E}\max_{i\le j} \alpha_{i}\xi_{i}=\mathbb{E}\max_{i\le j} \alpha_{i}^{*}\xi_{i}\ge\alpha_{j}^{*} \mathbb{E}\max_{i\le j}\xi_{j}\gtrsim\alpha_{j}^{*}\log(j+1)\nonumber,
\end{align}
where we used that the expectation of the maximum of $j$ i.i.d. Exponential variables with parameters $1$ is of order $\log (j+1)$ (Remark 2.1). It remains to take the maximum over $j\le n$.

\end{proof}

\section{Sub-Exponential Diagonal Matrix }

Let $Y$ be an $n\times n$ diagonal random matrix. Then, denote its entries by $Y_{1},\cdots, Y_{n}$, which are independent centered sub-Exponential random variables. In this section, we prove the following results for this remarkable example.

\begin{mytheo}
	The spectral norm of $Y$ satisfies
	\begin{align}
		\mathbb{E}\Vert Y\Vert \lesssim \max_{i\le n}\Vert Y_{i}\Vert_{\psi_{1}}^{*}\log(i+1)\nonumber
	\end{align}	
	and the following concentration inequality
	\begin{align}
		\mathbb{P}\{\Vert Y\Vert-C\max_{i}\Vert Y_{i}\Vert_{\psi_{1}}\log n>t\}\le \exp\left(-C_{1}\min(\frac{t}{\max_{i}\Vert Y_{i}\Vert_{\psi_{1}}}, \frac{t^{2}}{\max_{i}\Vert Y_{i}\Vert_{\psi_{1}}^{2}})\right),\nonumber
	\end{align}
where $\{\Vert Y_{i}\Vert_{\psi_{1}}^{*}\}$ is the descreasing rearrangement of $\{\Vert Y_{i}\Vert_{\psi_{1}}\}$, $t\ge 0$ and $C,  C_{1}$ are universal constants.
\end{mytheo}


\begin{proof}
	Let $\eta_{i}, i=1,\cdots, n$ be a sequence of random variables satisfying exponential distribution with parameter $1$. Then we have by Stirling Formular
	\begin{align}
		(\mathbb{E}\vert \eta_{i}\vert^{p})^{1/p}=(p!)^{1/p}\asymp p.\nonumber
	\end{align}

It follows from sub-Exponential properties
\begin{align}
	(\mathbb{E}\vert Y_{i}\vert^{p})^{1/p}\lesssim \beta_{i} p,\nonumber
\end{align}
where $\beta_{i}=\Vert Y_{i}\Vert_{\psi_{1}}$.

Note that $\Vert Y\Vert=\max_{i} \vert Y_{i}\vert$. By virtue of Lemma 2.5, we have
\begin{align}
	\mathbb{E}\Vert Y\Vert\lesssim \mathbb{E}\max_{i\le n}\beta_{i}\eta_{i}\asymp\max_{i\le n}\beta_{i}^{*}\log(i+1).\nonumber
\end{align}

By Markov's inequality, we have for $\lambda, t>0$
\begin{align}\label{3.1}
	&\mathbb{P}\{\max_{i}\vert Y_{i}\vert-C_{0}(\max_{i}\beta_{i})\mathbb{E}\max_{i}\eta_{i}>t\}\nonumber\\
	\le &e^{-\lambda t}\mathbb{E}\exp\left(\lambda \big(\max_{i}\vert Y_{i}\vert-C_{0}(\max_{i}\beta_{i})\mathbb{E}\max_{i}\eta_{i}\big)\right).
\end{align}

Using Lemma 2.5 again, we have
\begin{align}
	\mathbb{E}\exp(\lambda\max_{i}\vert Y_{i}\vert)\le \mathbb{E}\exp(\lambda C_{0}\max_{i}\beta_{i}\eta_{i})\le \mathbb{E}\exp \left(\lambda C_{0}(\max_{i}\beta_{i})\max_{i}\eta_{i}\right)\nonumber.
\end{align}

Recalling Remark 2.1, we have
\begin{align}
	\mathbb{E}\exp\big(\lambda C_{0} \max_{i}\beta_{i}(\max_{i}\eta_{i}-\mathbb{E}\max_{i}\eta_{i})\big)=\prod_{i\le n}\mathbb{E}\exp\left(\lambda C_{0} \max_{i}\beta_{i}\big(\frac{T_{i}-2}{2(n-i+1)}\big)\right),\nonumber
\end{align}
where $T_{1}, \cdots, T_{n}$ is a sequence of i.i.d. chi-square random variables with parameters $2$.

By virtue of Lemma 2.1, we have for $\lambda\le C_{2}/\max_{i}\beta_{i}$
\begin{align}
	\mathbb{E}\exp\left(\lambda C_{0}\max_{i}\beta_{i}\big(\frac{T_{i}-2}{2(n-i+1)}\big)\right)\le  \exp\left(\frac{C_{3}\lambda^{2}\max_{i}\beta_{i}^{2}}{(n-i+1)^{2}}\right).\nonumber
\end{align}

Then, (\ref{3.1}) is further bounded by
\begin{align}
	\exp(-\lambda t+C_{4}\lambda^{2}\max_{i}\beta_{i}^{2}),
\end{align}
where $C_{4}$ is a universal constant and $\lambda\le C_{2}/\max_{i}\beta_{i}$. To optimize this bound, let
\begin{align}
	\lambda=\min(\frac{C_{2}}{\max_{i}\beta_{i}}, \frac{t}{2C_{4}\max_{i}\beta_{i}^{2}}).\nonumber
\end{align}

Hence, we have
\begin{align}
	\mathbb{P}\big\{\max_{i}\vert Y_{i}\vert-C_{0}(\max_{i}\beta_{i})\mathbb{E}\max_{i}\eta_{i}>t\big\}\le \exp\left(-C_{1}\min\big(\frac{t}{\max_{i}\beta_{i}}, \frac{t^{2}}{\max_{i}\beta_{i}^{2}}\big)\right).\nonumber
\end{align}

The desired result follows from the expectation of $\max_{i}\eta_{i}$ is of order $\log n$.

\end{proof}

\section{Symmetric Sub-Exponential Random Matrix}

In this section, we shall prove our main result. First, we decompose an arbitrary point on the unit sphere $z$ into $l(n)$ parts. Then we find a set $\mathcal{M}$ satisfying the contraction inequality (\ref{contraction}). To control the supremum of $\langle Xx, x\rangle$ on the $\mathcal{M}$, we decompose $\langle Xx, x\rangle$ into two parts. One of them can be controlled by Theorem 3.1. For the other one, we get a tight bound for a fixed $x$ and then take a union bound over all $x$ in $\mathcal{M}$ to get the desired result.

Recall the definitions $b=\max_{i}\Vert X_{ii}\Vert_{\psi_{1}}\log n$ , $\sigma_{1}=\max_{i, j\le n}\Vert X_{ij}\Vert_{\psi_{1}}$ and $\sigma_{2}=\max_{i\neq j} \Vert X_{ij}\Vert_{\psi_{1}}$.
\begin{proof}[The Proof of Theorem 1.1]
	The spectral norm of symmetric matrix $X$ can be written as follows
	\begin{align}
		\Vert X\Vert=\sup_{z\in S^{n-1}}\langle Xz, z\rangle.\nonumber
	\end{align}
	We first give an index set $\mathcal{M}\subseteq S^{n-1}$ such that $\Vert X\Vert$ can be controlled by the maximum of the stochastic process $\{\langle Xx, x\rangle\}_{x\in \mathcal{M}}$.

	Pick an arbitrary $z\in S^{n-1}$. Let $l_{1}, l_{2},\cdots, l_{n}$ be such that $\vert z(l_{1})\vert\ge \vert z(l_{2})\vert\ge \cdots \ge \vert z(l_{n})\vert $.
	Define $l$ as the smallest integer such that
	\begin{align}\label{condition}
		\frac{n}{4^{l}}\log 96e\cdot4^{l}\le \sqrt{n}.
	\end{align}

We set
\begin{align}
	E_{0}=\{l(i)\}_{1\le i\le n/4^{l}}\nonumber
\end{align}
and
\begin{align}
	E_{1}=\{l(i) \}_{n/4<i\le n}, E_{2}=\{l(i) \}_{n/16<i\le n/4},\cdots, E_{l}=\{l(i) \}_{n/4^{l}<i\le n/4^{l-1}}.\nonumber
\end{align}
We now decompose $z$ into $\sum_{k=0}^{l}P_{E_{k}}z$, where $P_{E_{k}}$ is an orthogonal projection from $\mathbb{R}^{n}$ to $\mathbb{R}^{E_{k}}$ defined as before. 	

Due to $z\in S^{n-1}$, we observe that for every $k\ge 1$
\begin{align}
	\Vert P_{E_{k}}z\Vert_{\infty}\le \vert z(l_{s})\vert\le \sqrt{\frac{4^{k}}{n}},\nonumber
\end{align}
where $s=[n/4^{k}]$. Then for every $k\ge 1$, the vector $P_{E_{k}}z$ can be approximated by a vector from $\mathcal{N}(E_{k}, 4^{-k}, \sqrt{\frac{4^{k}}{n}})$ and the vector $P_{E_{0}}z$ can be approximated by a vector from $\mathcal{N}(E_{0}, 1/16, 1)$.

Hence, we set $\mathcal{M}=\mathcal{M}_{0}\cap 2B_{2}^{n}$. $\mathcal{M}_{0}$ consists of all vectors of the form $x=\sum_{k=0}^{l}x_{k}$, where the $x_{i}$ have disjoint supports and
\begin{align}
	x_{0}\in \bigcup_{\substack{E\subset \{1,\cdots, n\}\\ \vert E\vert\le a_{0}}}\mathcal{N}(E, 1/16, 1),\quad x_{k}\in \bigcup_{\substack{E\subset \{1,\cdots, n\}\\ \vert E\vert\le a_{k}}}\mathcal{N}(E, 4^{-k}, \sqrt{\frac{4^{k}}{n}}),\quad 1\le k\le l ,\nonumber
\end{align}
where
$
	a_{0}:=\vert E_{0}\vert\le n/4^{l}
$
and
$
	a_{k}:=\vert E_{k}\vert\le n/4^{k-1}
$ for $1\le k\le l$
satisfying $\sum_{k=0}^{l}a_{k}=n$.

For an arbitrary $z\in S^{n-1}$, there exists $x\in \mathcal{M}$, with a suitable representation $x=\sum_{k=0}^{l}x_{k}$, such that
\begin{align}
	\langle Xz, z \rangle&\le \langle Xx, x \rangle +\langle X(z-x), x \rangle+ \langle Xz, (z-x) \rangle\nonumber\\
	&\le \langle Xx, x \rangle +\langle X(z-x), (z+x) \rangle \nonumber\\
	&\le \langle Xx, x \rangle +\Vert X\Vert \cdot\vert z-x\vert \cdot\vert z+x\vert.\nonumber
\end{align}

Note that $\vert z+x\vert\le 3$ and
\begin{align}
	\vert (z-x)\vert=  (\sum_{k=0}^{l}\vert x_{k}-P_{E_{k}}z\vert^{2})^{1/2}\le 0.3.\nonumber
\end{align}

Thus we get the following significant contraction inequality
\begin{align}\label{contraction}
	\Vert X\Vert\le 10\sup_{x\in \mathcal{M}} \langle Xx, x \rangle.
\end{align}

Next, we turn to estimate $\sup_{x\in \mathcal{M}} \langle Xx, x \rangle$. Fix $x\in\mathcal{M}$ of the form $x=\sum_{k=0}^{l}x_{k}$ and let $F_{k}$ be the support of $x_{k}$. Denote the coordinates of $x$ by $x(i), i\le n$. Then
\begin{align}\label{16}
	\langle Xx, x \rangle=\sum_{i=1}^{n}\sum_{j=1}^{n}x(i)x(j)X_{ij}=\sum_{i=1}^{n}x(i)^{2} X_{ii}+D_{x},
\end{align}
where $D_{x}=\sum_{i\neq j}x(i)x(j)X_{ij}$.

Note that $\vert\sum_{i=1}^{n}x(i)^{2} X_{ii}\vert \le 4\max_{i}\vert X_{ii}\vert$. Hence, we have by virtue of Theorem 3.1
\begin{align}\label{end1}
	\mathbb{P}\{\sup_{x\in \mathcal{M}}\sum_{i=1}^{n}x(i)^{2}X_{ii}\ge C_{2}b+t\}\le C_{3}e^{-C_{4}t/\sigma_{1}},
\end{align}
where $C_{2}, C_{3}$ and $C_{4}$ are universal constants.

 If we get a similar estimate for $D_{x}$, we shall get the desired result by the union bound. To this aim, we split $D_{x}$ according to the structure of $x$. Define
\begin{align}
	D_{x}^{'}:=\sum_{k=0}^{l}\sum_{\substack{i, j\in F_{k}\\ i\neq j}}x(i)x(j)X_{ij}\nonumber
\end{align}
and
\begin{align}
	D_{x}^{''}:=\sum_{k=0}^{l}\sum_{\substack{i\in F_{k}\\ j\notin F_{k}}}x(i)x(j)X_{ij}=2\sum_{k=1}^{l}\sum_{i\in F_{k}}x(i)\sum_{r\in G_{k}} \sum_{j\in F_{r}}x(j)X_{ij},\nonumber
\end{align}
where $G_{k}=\{0,k+1,k+2,\cdots, l\}.$ Note that $D_{x}=D_{x}^{'}+D_{x}^{''}$.

We first estimate $D_{x}^{'}$.  For every $k$, we have
\begin{align}
	2^{|F_{k}|-2}\sum_{\substack{i, j\in F_{k}\\ i\neq j}}x(i)x(j)X_{ij}&=\sum_{E\subseteq F_{k}}\sum_{i\in E}\sum_{j\in E^{c}}x(i)x(j)X_{ij}\nonumber\\
	&\le 2^{|F_{k}|}\max_{E\subseteq F_{k}}\sum_{i\in E}\sum_{j\in E^{c}}x(i)x(j)X_{ij}\nonumber,
\end{align}
which implies that there exists  $F_{k}^{'}$, the subsets of $F_{k}$, such that
\begin{align}\label{D_{x}}
	D_{x}^{'}\le& 4\sum_{k=0}^{l}\sum_{i\in F_{k}^{'}}\sum_{j\in F_{k}\backslash  F_{k}^{'}}x(i)x(j)X_{ij}\nonumber\\
	\le &4\sup_{\substack{F\subset \{1,\cdots, n\}\\ \vert F\vert\le a_{0}}}\sup_{E\subset F}\sup_{v\in \mathcal{N}(F,1/16,1)}\sum_{i\in E}\sum_{j\in F\backslash E}v(i)v(j)X_{ij}\nonumber\\
	+& 4\sum_{k=1}^{l}\sup_{\substack{F\subset \{1,\cdots, n\}\\ \vert F\vert\le a_{k}}}\sup_{E\subset F}\sup_{v\in \mathcal{N}(F,4^{-k},\sqrt{4^{k}/n})}\sum_{i\in E}\sum_{j\in F\backslash E}v(i)v(j)X_{ij}.
\end{align}

For $1\le k\le l$, let $L_{k}=\frac{8n}{4^{k}}t\log (6e\cdot4^{2k})$ for $t\ge 1$. For $F\subseteq \{1,\cdots, n\}$ with $\vert F\vert\le a_{k}$, $E\subseteq F$, and $v\in \mathcal{N}(F, 4^{-k}, \sqrt{4^{k}/n} )$, we define the following set
\begin{align}
	\Omega(F, E, v):=\biggl\{ \sum\limits_{i\in E}\sum\limits_{j\in F\backslash E}v(i)v(j)X_{ij} > \sigma_{2}\sqrt{4^{k}/n} L_{k}\biggr\}.\nonumber
\end{align}

Note that $\Vert v\Vert_{\infty}\le \sqrt{4^{k}/n}\le 1$. We have
\begin{align}
		\sum\limits_{i\in E} \sum\limits_{j\in F\backslash E}\vert v(i)v(j)X_{ij}\vert\le \sqrt{4^{k}/n}  \sum\limits_{i\in E}\sum\limits_{j\in F\backslash E}\vert v(i)X_{ij}\vert,\nonumber
\end{align}
which implies that
\begin{align}\label{12}
	\mathbb{P}\{\Omega(F, E, v)\}&\le \mathbb{P}\{\sum\limits_{i\in E}\sum\limits_{j\in F\backslash E}\vert v(i)X_{ij}\vert>\sigma_{2}L_{k}\}\nonumber\\
	&\le e^{-L_{k}}\mathbb{E}\exp\big(\sum_{i\in E}\sum\limits_{j\in F\backslash E}\frac{\vert v(i)X_{ij}\vert}{\sigma_{2}}\big).
\end{align}

 Due to the sub-Exponential properties, we have
\begin{align}
	\mathbb{E}\exp\big(\sum_{i\in E}\sum_{j\in F\backslash E}\frac{\vert v(i)X_{ij}\vert}{\sigma_{2}}\big)&=\prod_{i\in E}\prod_{j\in F\backslash E}\mathbb{E}\exp (\frac{\vert v(i)X_{ij}\vert}{\sigma_{2}})\nonumber\\
	&\le 2^{\vert E\vert} \le 2^{a_{k}}. \nonumber
\end{align}
Therefore by the union bound and the fact $|\mathcal{N}(F, \varepsilon, \alpha)|\le (3/\varepsilon)^{|F|}$, we have
\begin{align}\label{PP}
	\mathbb{P}\biggl\{&\sup_{\substack{F\subset \{1,\cdots, n\}\\ \vert F\vert\le a_{k}}}\sup_{E\subset F}\sup_{v\in \mathcal{N}(F,4^{-k}, \sqrt{\frac{4^{k}}{n}})}\sum\limits_{i\in E}\sum\limits_{j\in F\backslash E}\vert v(i)X_{ij}\vert>\sigma_{2}\sqrt{4^{k}/n}L_{k}\biggr\}\nonumber\\
	&\le \sum_{m\le a_{k}}\binom{n}{m}(3\cdot 4^{k})^{a_{k}}\sup_{F,E,v} \mathbb{P}(\Omega(F,E,v))\nonumber\\
	&\le \sum_{m\le n/4^{k-1}}\binom{n}{m}(6\cdot 4^{k})^{n/4^{k-1}}e^{-L_{k}}\le (6e\cdot 4^{2k})^{n/4^{k-1}}e^{-L_{k}}\nonumber\\
	&=\exp\big(\frac{4n}{4^{k}}\log (6e\cdot 4^{2k})-L_{k}\big)\le e^{-L_{k}/2},
\end{align}
where the third inequality is due to $\sum_{m=1}^{k}\binom{n}{m}\le (\frac{en}{k})^{k}$.

For $k=0$, let $L=2\sqrt{n}t$ for $t\ge 1$. Due to the definition of $l$, we have $L\ge \frac{2nt}{4^{l}}\log (96e\cdot 4^{l})$. Following the same line as above, we have

\begin{align}\label{explain}
	\mathbb{P}\biggl\{&\sup_{\substack{F\subset \{1,\cdots, n\}\\ \vert F\vert\le a_{0}}}\sup_{E\subset F}\sup_{v\in \mathcal{N}(F,1/16,1)}\sum_{i\in E}\sum_{j\in F\backslash E}v(i)v(j)X_{ij}> \sigma_{2} L\biggr\}
	\le e^{-\sqrt{n}t}.
\end{align}

Recall the equation (\ref{D_{x}}). Then we obtain by the union boud
\begin{align}\label{17}
	\mathbb{P}\bigg\{\sup_{x\in\mathcal{M}}D_{x}^{'}> 4\sigma_{2}L+4\sigma_{2}\sum_{k=1}^{l}\sqrt{\frac{4^{k}}{n}}L_{k}\bigg\}\le& \exp(-L/2)+\sum_{k=1}^{l}\exp(-L_{k}/2)\nonumber\\
	= &\exp(-\sqrt{n}t)+\sum_{k=1}^{l}\exp\left(-\frac{4nt}{4^{k}}\log (3e\cdot4^{2k})\right)\nonumber\\	
	\le& \exp(-\sqrt{n}t)+l\exp\left(-\frac{4nt}{4^{l}}\log (3e\cdot4^{2l})\right).
\end{align}

Observe that
\begin{align}
	\sum_{k=1}^{l}\sqrt{\frac{4^{k}}{n}}L_{k}=\sum_{k=1}^{l}\sqrt{\frac{4^{k}}{n}}\frac{4nt}{4^{k}}\log (3e\cdot4^{2k})\lesssim\sqrt{n}t.\nonumber
\end{align}

Recall the choice of $l$ and note that $l\ge 2$. We have
$$\sqrt{n}< \frac{n}{4^{l-1}}\log 96e\cdot4^{l-1}<\frac{4n}{4^{l}}\log (3e\cdot4^{2l})$$
and $l\lesssim \sqrt{n}$.
Then, we obtain by adjusting the constant
\begin{align}\label{18}
	&\mathbb{P}\{\sup_{x\in\mathcal{M}}D_{x}^{'}>C_{10}\sigma_{2} \sqrt{n}t\}\nonumber\\
	\le &e^{-\sqrt{n}t}+le^{-\sqrt{n}t}\le C_{11}\sqrt{n}e^{-\sqrt{n}t},
\end{align}
where $C_{10}, C_{11}$ are universal constants and $t\ge 1$.

Next, we shall follow the same line to  estimate $D_{x}^{''}$. Consider $\mathcal{M}_{k}=\mathcal{M}_{k}^{'}\cap 2B_{2}^{n}$ for $1\le k\le l$, where $\mathcal{M}_{k}^{'}$ consists of all vectors of the form $x=x_{0}+\sum_{s=k+1}^{l}x_{s}$, where the $x_{i} (i=0, k+1, k+2,\cdots, l)$ have disjoint supports and
\begin{align}
	x_{0}\in \bigcup_{\substack{E\subset \{1,\cdots, n\}\\ \vert E\vert\le a_{0}}}\mathcal{N}(E, 1/16, 1),\quad x_{s}\in \bigcup_{\substack{E\subset \{1,\cdots, n\}\\ \vert E\vert\le a_{s}}}\mathcal{N}(E, 4^{-s}, \sqrt{\frac{4^{s}}{n}})\nonumber
\end{align}
for $s\ge k+1$.

Note that $\mathcal{M}_{k}\subset2B_{2}^{n}$ and
\begin{align}
	\vert \mathcal{M}_{k}\vert&\le 48^{n/4^{l}}\big(\sum_{u_{0}\le n/4^{l}}{n\choose u_{0}}\big)\prod_{s=k+1}^{l}\big((3\cdot4^{s})^{n/4^{s-1}}(\sum_{u_{s}\le n/4^{s-1}}{n\choose u_{s}})\big)\nonumber\\
	&\le ( 48e\cdot4^{l})^{n/4^{l}}\prod_{s=k+1}^{l}(3e\cdot 4^{2s})^{n/4^{s-1}}\nonumber\\
	&\le \exp\big(\frac{n}{4^{l}}\log (48e\cdot 4^{l}) +\sum_{s=k+1}^{l}\frac{4n}{4^{s}}\log (3e\cdot 4^{2s})\big)\nonumber\\
	&\le \exp\big(\sum_{s=k+1}^{l+1}\frac{4n}{4^{s}}\log (3e\cdot 4^{2s})\big)\nonumber\\
	&\le \exp\Big(\frac{4n}{4k}\big(\sum_{s=1}^{l-k+1}\frac{\log(3e\cdot 4^{2k})}{4^{s}}+\sum_{s=1}^{l-k+1}\frac{\log(4^{2s})}{4^{s}}\big)\Big)\le \exp\big(\frac{4n}{4k}\log(48e\cdot 4^{2k})\big).\nonumber
\end{align}

We also observe that
\begin{align}
	D_{x}^{''}=&2\sum_{k=1}^{l}\sum_{i\in F_{k}} x(i) \sum_{r\in G_{k}}\sum_{j\in F_{r}}x(j)X_{ij}\nonumber\\
	\le &2\sum_{k=1}^{l}\sup_{\substack{F\subset \{1,\cdots, n\}\\ \vert F\vert\le a_{k}}}\sup_{u\in \mathcal{N}(F,4^{-k},\sqrt{\frac{4^{k}}{n}})}\sup_{v\in \mathcal{M}_{k}}\sum_{i\in F} u(i)\sum_{j\notin F}v(j)X_{ij}.\nonumber
\end{align}

Similar to (\ref{PP}), we have for fixed $k, 1\le k\le l$
\begin{align}
	&\mathbb{P}\{ \sup_{\substack{F\subset \{1,\cdots, n\}\\ \vert F\vert\le a_{k}}}\sup_{u\in \mathcal{N}(F,4^{-k},\sqrt{\frac{4^{k}}{n}})}\sup_{v\in \mathcal{M}_{k}}\sum_{i\in F} u(i)\sum_{j\notin F}v(j)X_{ij}>\sigma_{2}\sqrt{4^{k}/n}L(k)\}\nonumber\\
	\le& |\mathcal{M}_{k}|(3e\cdot4^{4k})^{\frac{4n}{4^{k}}}e^{-L(k)}\le e^{-L_{k}/2},\nonumber
\end{align}
where $L(k)=\frac{16nt}{4^{k}}\log(48e\cdot4^{2k})$ and $t\ge 1$.

It follows from the union bound
\begin{align}
	&\mathbb{P}\big\{D_{x}^{''}>2\sigma_{2}\sum_{k=1}^{l}\sqrt{\frac{4^{k}}{n}}L(k) \big\}
	\le \sum_{k=1}^{l}e^{-L_{k}/2}\nonumber\\
	\le &\sum_{k=1}^{l}\exp \big(\frac{-8nt}{4^{k}}\log(48e\cdot4^{2k})\big)\le l\exp \big(\frac{-8nt}{4^{l}}\log(48e\cdot4^{2l})\big)\nonumber.
\end{align}
Similar to (\ref{18}), we obtain by adjusting the constant
\begin{align}
	\mathbb{P}\{\sup_{x\in \mathcal{M}}D_{x}^{''}>C_{12}\sigma_{2} \sqrt{n}t\}\le C_{13}\sqrt{n}e^{-\sqrt{n}t},
\end{align}
where $C_{12}$ and $C_{13}$ are universal constants and $t\ge 1$.

Since $D_{x}=D_{x}^{'}+D_{x}^{''}$, then
\begin{align}
	\mathbb{P}\{\sup_{x\in \mathcal{M}}D_{x}>C_{14}\sigma_{2} \sqrt{n}t\}\le C_{15}\sqrt{n}e^{-\sqrt{n}t},\nonumber
\end{align}
where $C_{14}$ and $C_{15}$ are universal constants and $t\ge 1$. Note that $\sqrt{n}\le e^{\sqrt{n}/2}$, then we have by taking $t=1+u/\sqrt{n}$
\begin{align}\label{end2}
	\mathbb{P}\{\sup_{x\in \mathcal{M}}D_{x}>C_{16}\sigma_{2} \sqrt{n}+C_{16}\sigma_{2}u\}\le C_{17}e^{-u},
\end{align}
where $C_{16}, C_{17}$ are universal constants and $u\ge 0$.

By virtue of (\ref{end1}) and (\ref{end2}), we have
\begin{align}
	\mathbb{P}\{\Vert X\Vert\ge C_{18}(b+\sigma_{2}\sqrt{n})+2u\}
	\le &\mathbb{P}\{4\max_{i} X_{ii} \ge C_{18}b+u\}\nonumber\\
	&+\mathbb{P}\{\sup_{x\in\mathcal{M}}D_{x} \ge C_{18}\sqrt{n}+\sigma_{2}u\}\nonumber,
\end{align}
where $C_{18}$ is a universal constant and $u\ge 0$.
Then we get the desired result (\ref{5}) by adjusting the constant.
\end{proof}

\begin{myrem}
	Throught the proof, an important observation is that, if for a vector $x$ there is a simultaneous control of the size of support and its $l_{\infty}$-norm, then one can estimate $\left<Xx, x\right>$ with large probability, see $(4.7)$. It is therefore natural to decompose the vectors on the unit sphere into $l(n)$ parts admitting such a simultaneous control as above. At last, using the union bound for the $l(n)$ parts yields the desired result.
\end{myrem}

\section{Extensions and discussion}
\subsection{Exponentially Decaying Variables}

We have phrased our results in terms of sub-Exponential entries. In this section, we develop some of our  results into more general distributions of the entries. In particular, we consider the class of distributions whose tail decay is of the type $\exp(-t^{\alpha})$ or faster, where $0< \alpha \le 2$. Note that $\alpha =2$ corresponds to sub-Gaussian distributions and $\alpha =1$ to sub-Exponential distributions.

Let us begin by considering the following lemma, which is Exercise 2.7.3 in \cite{highdimension}.

\begin{mylem}
	Let $\eta$ be a random variable. Then the following properties are equivalent for $0<\alpha\le 2$.
	\begin{enumerate}
		\item $\mathbb{P}\{\vert \eta\vert\ge t \}\le 2\exp(-t^{\alpha}/K_{1}^{\alpha})$ for all $t\ge 0$;
		\item $(\mathbb{E}\vert \eta\vert^{p})^{1/p}\le K_{2}p^{\alpha^{-1}}$ ;
		\item $\mathbb{E}\exp(\lambda^{\alpha}\vert\eta\vert^{\alpha})\le \exp(K_{3}^{\alpha}\lambda^{\alpha})$ for all $\lambda$ such that $0\le \lambda\le \frac{1}{K_{3}}$;
		\item $\mathbb{E}\exp(\vert\eta\vert^{\alpha}/K_{4}^{\alpha})\le2$ for a constant $K_{4}>0$.
	\end{enumerate}
\end{mylem}
\begin{proof}
	(1)$\Rightarrow$(2): Without loss of generality, we assume property (1) holds with $K_{1}=1$. By virtue of (\ref{1}), we have
	\begin{align}
		\mathbb{E}\vert \eta\vert^{p}&=\int_{0}^{\infty}\mathbb{P}\{ \vert \eta\vert^{p}\ge u\}\, du\nonumber\\
		&=\int_{0}^{\infty}\mathbb{P}\{ \vert \eta\vert\ge t\}pt^{p-1}\, dt\nonumber\\
		&\le \int_{0}^{\infty}2e^{-t^{\alpha}}pt^{p-1}\, dt\nonumber\\	
		&=\frac{2p}{\alpha}\Gamma(\frac{p}{\alpha})\le2(\frac{p}{\alpha})^{p/\alpha +1} \nonumber,
	\end{align}
where the last inequality is due to $\Gamma(x)\le x^{x}$ by Stirling's approximation.	Taking the $p$th root yields property (2).

(2)$\Rightarrow$(3): Assume that property (2) holds with $K_{2}=1$. We obtain by the Taylor series expansion of the exponential function
\begin{align}\label{0.5.1}
	\mathbb{E}\exp(\lambda^{\alpha}\vert \eta\vert^{\alpha})=1+\sum_{p=1}^{\infty}\frac{\lambda^{\alpha p}\mathbb{E}\vert \eta\vert^{\alpha p}}{p!}.
\end{align}
Note that $\mathbb{E}\vert \eta\vert^{\alpha p}\le (\alpha p)^{p}$ by property (2) and $p!\ge (p/e)^{p}$ by Stirling's approximation. Hence, (\ref{0.5.1}) is bounded by
\begin{align}
	1+\sum_{p=1}^{\infty}(\alpha e \lambda^{\alpha})^{p}=\frac{1}{1-\alpha e \lambda^{\alpha}},\nonumber
\end{align}
provided that $\alpha e \lambda^{\alpha}<1$, in which case the geometric series above converges. Note that, $1/(1-x)\le e^{2x}$ when $0\le x\le 1/2$. Hence, we have for $0\le \lambda \le (2\alpha e)^{-\alpha^{-1}}$.
\begin{align}
	\mathbb{E}\exp (\lambda^{\alpha}\vert \eta\vert^{\alpha})\le \exp(2\alpha e\lambda^{\alpha}).\nonumber
\end{align}

(3)$\Rightarrow$(4): This result is trivial.

(4)$\Rightarrow$(1): Assume that property (4) holds with $K_{4}=1$.
\begin{align}
	\mathbb{P}\{\vert \eta\vert>t\}&=\mathbb{P}\{ e^{\vert \eta\vert^{\alpha}}\ge e^{t^{\alpha}}\}\nonumber\\
	&\le e^{-t^{\alpha}}\mathbb{E}e^{\vert \eta\vert^{\alpha}}\le 2e^{-t^{\alpha}}.\nonumber
\end{align}

We have finished the proof of Lemma 5.1.
\end{proof}

\begin{myrem}
	When $1\le \alpha\le 2$, the following quantity
	\begin{align}\label{5.1}
	\Vert \eta\Vert_{\psi_{\alpha}}:=	\inf \{t>0:  \mathbb{E}\exp(\vert\eta\vert^{\alpha}/t^{\alpha})\le2\}
	\end{align}
can define a norm of random variable $\eta$. When $0<\alpha<1$, (\ref{5.1}) is not a norm due to that
$\exp(x^{\alpha})$ is not convex on the interval $(0, \infty)$.
\end{myrem}

The next lemma is an extension of Lemma 2.6.
\begin{mylem}\label{5.2}
	Suppose that $\xi_{1}, \xi_{2}, \cdots, \xi_{n}$ are  independent random variables with $\Vert\xi_{i}\Vert_{\psi_{\alpha}}=1$ satisfying
	\begin{align}
		(\mathbb{E}\vert \xi_{i}\vert^{p})^{1/p}\asymp p^{\alpha^{-1}},
	\end{align}
where $0<\alpha\le 2$. Let $c_{1},\cdots, c_{n}$ be a sequence of positive constants. We have
\begin{align}
	\mathbb{E}\max_{i\le n}\vert c_{i}\xi_{i}\vert\asymp \max_{i\le n}c_{i}^{*}\log^{\alpha^{-1}}(i+1),\nonumber
\end{align}
where $\{c_{i}^{*}\}$ is decreasing rearrangement of $\{c_{i}\}$.
\end{mylem}
\begin{myrem}
	Condition (\ref{5.2}) implies that
	\begin{align}
		c_{1}(\alpha)e^{-t^{\alpha}/c_{1}(\alpha)}\le \mathbb{P}\{ \vert \xi_{i}\vert\ge t\}\le c_{2}(\alpha)e^{-t^{\alpha}/c_{2}(\alpha)},\nonumber
	\end{align}
which has been proved in Lemma 4.6 of \cite{inventions}.
\end{myrem}
\begin{proof}
	Let $\eta_{1},\cdots,\eta_{n}$ be a sequence of independent standard normal random variables. Then $\vert\eta_{i}\vert^{2/\alpha}$ satisfies
	\begin{align}
		(\mathbb{E}\vert\eta_{i}\vert^{2p/\alpha} )^{1/p}\asymp p^{\alpha^{-1}}.\nonumber
	\end{align}
Note that $2/\alpha \ge 1$, hence we have
\begin{align}
	\mathbb{E}\max_{i\le n}c_{i}\vert\eta_{i}\vert^{2/\alpha}\ge (\mathbb{E}\max_{i\le n}c_{i}^{\alpha/2}\vert\eta_{i}\vert)^{2/\alpha}\asymp \max_{i\le n}c_{i}^{*}\log^{\alpha^{-1}}(i+1).\nonumber
\end{align}

It follows from Lemma 2.5
\begin{align}
\mathbb{E}\max_{i\le n}\vert c_{i}\xi_{i}\vert\gtrsim\max_{i\le n}c_{i}^{*}\log^{\alpha^{-1}}(i+1).\nonumber
\end{align}

Next, we shall prove the opposite inequality by an induction argument. Let $\eta_{1}^{'},\cdots, \eta_{n}^{'}$ be a sequence of copies of $\xi_{i}$.

For $1\le \alpha\le 2$, we have
\begin{align}
	(\mathbb{E}\vert \eta_{i}\vert^{p}\vert\eta_{i}^{'}\vert^{(2-\alpha)p/\alpha} )^{1/p}\asymp p^{\alpha^{-1}}.\nonumber
\end{align}

Note that $0\le (2-\alpha)/\alpha\le 1$, hence we have
\begin{align}\label{5.3}
\mathbb{E}\max_{i\le n}c_{i}	\vert \eta_{i}\vert\vert\eta_{i}^{'}\vert^{(2-\alpha)/\alpha}&=\mathbb{E}\Big(\mathbb{E}\big(\max_{i\le n}c_{i}	\vert \eta_{i}\vert\vert\eta_{i}^{'}\vert^{(2-\alpha)/\alpha}\vert \eta_{1}^{'},\cdots, \eta_{n}^{'}\big)\Big)\nonumber\\
&\lesssim\mathbb{E}\max_{i\le n}c_{i}\log^{1/2}(i+1)\vert\eta_{i}^{'}\vert^{(2-\alpha)/\alpha}\nonumber\\
&\le \Big(\mathbb{E}\max_{i\le n}\big(c_{i}\log^{1/2}(i+1)\big)^{\alpha/(2-\alpha)}\vert\eta_{i}^{'}\vert\Big)^{(2-\alpha)/\alpha}\nonumber\\
&\lesssim \max_{i\le n}c_{i}\log^{\alpha^{-1}}(i+1).
\end{align}
Using Lemma 2.5 again, we get the opposite inequality for $1\le \alpha\le 2$.

Assume that the desired result is true for $2/(k+1)\le \alpha\le 2$. Then for $2/(k+2)\le \alpha\le 2/(k+1)$, we consider the following random variable $$\vert \eta_{i}\vert^{k+1}\vert \eta_{i}^{'}\vert^{(2-(k+1)\alpha)/\alpha}.$$ Note that $0\le (2-(k+1)\alpha)/\alpha\le 1$, hence we get the oppsite inequality for $2/(k+2)\le \alpha\le 2/(k+1)$ following the same line in (\ref{5.3}).

Hence, we have for $0<\alpha\le 2$
\begin{align}
	\mathbb{E}\max_{i\le n}\vert c_{i}\xi_{i}\vert\lesssim \max_{i\le n}c_{i}\log^{\alpha^{-1}}(i+1).\nonumber
\end{align}

By permutation invariance, we have
\begin{align}
	\mathbb{E}\max_{i\le n}\vert c_{i}\xi_{i}\vert=\mathbb{E}\max_{i\le n}\vert c_{i}^{*}\xi_{i}\vert\lesssim \max_{i\le n}c_{i}^{*}\log^{\alpha^{-1}}(i+1),\nonumber
\end{align}
which concludes the desired result.
\end{proof}

Next, we discuss the case for $1\le \alpha \le 2$. Let $Y$ be an $n\times n$ diagnal random matrix. Denote its entries by $Y_{1},\cdots, Y_{n}$ which are independent centered random variables satisfying for $p\ge 1$
\begin{align}
	(\mathbb{E}\vert Y_{i}\vert^{p})^{1/p}\asymp \Vert Y_{i}\Vert_{\psi_{\alpha}}p^{\alpha^{-1}}\nonumber.
\end{align}
\begin{mytheo}
	The spectral norm of $Y$ satisfies the following concentration inequality
	\begin{align}
		\mathbb{P}\{\Vert Y\Vert\ge C(\alpha)\max_{i\le n}\Vert Y_{i}\Vert_{\psi_{\alpha}}\log^{\alpha^{-1}}n+t\}\le C_{0}\exp(-C_{1}t/\max_{i}\Vert Y_{i}\Vert_{\psi_{\alpha}}),\nonumber
	\end{align}
where $C(\alpha)$ is a positive constant independent of $n$ but depending on $\alpha$ and $C_{0}, C_{1}$ are universal constants.
\end{mytheo}

\begin{proof}

Let $\xi_{1},\cdots, \xi_{n}$ be a sequence of independent exponential random variables with parameter $1$. Then we have
\begin{align}
	\mathbb{E}(\vert\xi_{i}\vert^{p/\alpha})^{1/p}\asymp p^{\alpha^{-1}}.\nonumber
\end{align}

Let $\beta_{i}=\Vert Y_{i}\Vert_{\psi_{\alpha}}$. Then, we have by Markov's inequality
\begin{align}\label{5.4}
	&\mathbb{P}\big\{\max_{i}\vert Y_{i}\vert-2^{\alpha^{-1}}(\max_{i}\beta_{i})(\mathbb{E}\max_{i}\xi_{i})^{\alpha^{-1}}\ge t\big\}\nonumber\\
	\le &e^{-\lambda t}\mathbb{E}\exp\Big(\lambda\big(\max_{i}\vert Y_{i}\vert-2^{\alpha^{-1}}(\max_{i}\beta_{i})(\mathbb{E}\max_{i}\xi_{i})^{\alpha^{-1}}\big)\Big)\nonumber\\
	\le &e^{-\lambda t}\mathbb{E}\exp\Big(\lambda\max_{i}\beta_{i}\big(\max_{i} \xi_{i}^{\alpha^{-1}}-2^{\alpha^{-1}}(\mathbb{E}\max_{i}\xi_{i})^{\alpha^{-1}}\big)\Big),
\end{align}
where the second inequality is due to Lemma 2.5.

 Note that
 \begin{align}
 	\max_{i}\xi_{i}^{\alpha^{-1}}\le 2^{\alpha^{-1}}\big(\vert \max_{i}\xi_{i}-\mathbb{E}\max_{i}\xi_{i}\vert^{\alpha^{-1}}+(\mathbb{E}\max_{i}\xi_{i})^{\alpha^{-1}}\big).\nonumber
 \end{align}

Hence, (\ref{5.4}) is further bounded by
\begin{align}
	&e^{-\lambda t}\mathbb{E}\exp\big(\lambda(\max_{i}\beta_{i})\vert2\max_{i} \xi_{i}-2\mathbb{E}\max_{i}\xi_{i}\vert^{\alpha^{-1}}\big)\nonumber\\
	\le &e^{-\lambda t}\mathbb{E}\exp\big(2\lambda(\max_{i}\beta_{i})\vert\max_{i} \xi_{i}-\mathbb{E}\max_{i}\xi_{i}\vert\big)\nonumber\\
	\le & e^{-\lambda t} e^{-C_{2}\max_{i}\beta_{i}\lambda},\nonumber
\end{align}
where the first inequality is due to Lemma 2.5, the second inequality follows from Remark 2.1, $C_{2}$ is a universal constant, and $0<\lambda\le (C_{2}\max_{i}\beta_{i})^{-1}$.

Let $\lambda =(C_{2}\max_{i}\beta_{i})^{-1}$. Then we get the desired result.
\end{proof}

Let $X$ be an $n\times n$ symmetric random matrix with independent but non-identically distributed centered random entries satisfying for all $p\ge 1$
\begin{align}
	(\mathbb{E}\vert X_{ij}\vert^{p})^{1/p}\asymp\Vert X_{ij}\Vert_{\psi_{\alpha}}p^
	{\alpha^{-1}}, \nonumber
\end{align}
where $1\le \alpha\le 2$. For convenience, we define
$$ b(\alpha):=\max_{i\le n}\Vert X_{ii}\Vert_{\psi_{\alpha}}\log^{\alpha^{-1}}n,$$
$\sigma_{1}(\alpha):=\max_{ij}\Vert X_{ij}\Vert_{\psi_{\alpha}}$ and $\sigma_{2}(\alpha):=\max_{i\neq j}\Vert X_{ij}\Vert_{\psi_{\alpha}}$.
\begin{mytheo}
	The spectral norm of $X$ satisfies
	\begin{align}
		\mathbb{P}\{\Vert X\Vert\ge C_{3}(\alpha)(b(\alpha)+\sigma_{2}(\alpha)\sqrt{n})+t\}\le C_{4}\exp(-\frac{C_{5}t}{\sigma_{1}(\alpha)})\nonumber,
	\end{align}
where $C_{3}(\alpha) $ is a positive constant independent of $n$ but depending on $\alpha$ and $C_{4}, C_{5}$ are universal constants.
\end{mytheo}

	We remark that the chaining method we used to approximate the sub-Exponential case is unsuitable for heavier tail distributions. Hence, we can not  easily generalize our main result to the cases $0<\alpha<1$.
	
\begin{proof}
	To prove Theorem 5.2, we only need to get a similar bound to (\ref{12}), that is, to estimate the following quantity,
	$
		\mathbb{E}e^{\vert \xi\vert},\nonumber
	$
where $(\mathbb{E}\vert\xi\vert^{p})^{1/p}\asymp p^{1/\alpha}$ and $\Vert\xi\Vert_{\psi_{\alpha}}\le 1$.  In fact,

\begin{align}
	\mathbb{E}e^{\vert \xi\vert}&= \mathbb{E}e^{\vert \xi\vert}\mathbb{I}_{\{\vert \xi\vert\le 1\}}+\mathbb{E}e^{\vert \xi\vert}\mathbb{I}_{\{\vert \xi\vert> 1\}}\nonumber\\
	&\le e+\mathbb{E}\exp(\vert \xi\vert^{\alpha^{-1}})\mathbb{I}_{\{\vert \xi\vert> 1\}}\le 2+e.\nonumber
\end{align}
Then, we can get the desired result by adjusting some parameters.
\end{proof}

\subsection{Discussion} Our main result, Theorem 1.1, gives a tail bound of $\Vert X\Vert$. By virtue of (\ref{1}), we have
\begin{align}
	\mathbb{E}\Vert X\Vert\lesssim \max_{i\le n}\Vert X_{ii}\Vert_{\psi_{1}}\log n+\max_{i\neq j}\Vert X_{ij}\Vert_{\psi_{1}}\sqrt{n}.\nonumber
\end{align}
Obviously, this bound contains less information about the structure matrix. An interesting question is how to optimize this bound in concentration inequalities.

van Handel and Bandeira \cite{aop} obtained a nice bound (\ref{3}) for $\mathbb{E}\Vert X\Vert$ using the known combinatorial method based on the bounding trace of high powers of $X$. Besides, van Handel \cite{tams} proved an important dimension-free bound by Slepian's comparison inequality. Finally, based on the results of \cite{aop,tams}, authors of \cite{inventions} got a sharp bound for $\mathbb{E}\Vert X\Vert$.

Unfortunately, the combinatorial method used in \cite{aop} is unsuitable for what we care about, i.e.,  concentration inequalities. In particular, the quantity $\text{Tr}[X^{2p}]$ can be expanded as
\begin{align}\label{5.5}
	\text{Tr}[X^{2p}]=\sum_{u_{1},\cdots,u_{2p}\in [n]}X_{u_{1}u_{2}}X_{u_{2}u_{3}}\cdots X_{u_{2p}u_{1}},
\end{align}
where $([n], E_{n})$ is the complete graph on $n$ points. Note that $\mathbb{E}\xi^{i}=0$ when $\xi$ is symmetric and $i$ is odd. Hence, we can eliminate many summands in (\ref{5.5}) when considering the quantity $\mathbb{E}\text{Tr}[X^{2p}]$. Nevertheless, for the tail bounds, this property does not exist.

     In the progress of investigating the upper bound of $\mathbb{E}\Vert X\Vert$, authors \cite{aop,tams} also study the lower bound to obtain more information about $\Vert X\Vert$, which inspires us to consider the lower tail bounds, i.e., the lower bounds of $\mathbb{P}\{\Vert X\Vert>C(n, t, X_{ij})\}$.

     For Gaussian processes $(Y_{t})_{t\in T}$ and $(Z_{t})_{t\in T}$, Slepian's comparison inequality (Theorem 7.2.1 in \cite{highdimension}) states that if the process $(Y_{t})_{t\in T}$ grows faster (in terms of the magnitude of the increments), the further it gets, i.e., $$\mathbb{P}\{\sup_{t\in T}X_{t}\ge u\}\le \mathbb{P}\{\sup_{t\in T}Y_{t}\ge u\}.$$ Although Slepian's inequality is an excellent tool to study the lower tail bounds of $\Vert X\Vert$, we struggle to generalize this result to sub-exponential cases. We are  still considering how to get proper lower tail bounds of $\Vert X\Vert$ in our setting.



\vskip2mm
{\bf Acknowledgements } The authors would like to express their gratitude to the editor and anonymous referee for careful reading and constructive comments.


\end{document}